\documentclass[a4paper,12pt]{amsart}
\usepackage{latexsym, amsthm, amscd, graphicx, euscript, float, subfig}
\usepackage[english]{babel}
\usepackage{ifthen}
\usepackage{amsfonts}
\usepackage{amsthm}
\usepackage{amsmath,amssymb,graphicx}

\numberwithin{equation}{section}

\newtheorem{thm}{Theorem}[section]

\newtheorem{lem}[thm]{Lemma}
\newtheorem{prop}[thm]{Proposition}

\theoremstyle{definition}

\numberwithin{equation}{section}

\newcounter{alphabet}
\newcounter{tmp}
\newenvironment{Thm}[1][]{\refstepcounter{alphabet}%
\bigskip%
\noindent%
{\bf Theorem \Alph{alphabet}}%
\ifthenelse{\equal{#1}{}}{}{ (#1)}%
{\bf .}
\itshape}{\vskip 8pt}

\renewcommand{\arg}{{\mathroman{arg}\,}}
\newcommand{\B}{{\mathcal B}}

\newcommand{\Co}{{\mathcal{C}o}}
\newcommand{\C}{{\mathbb C}}
\newcommand{\D}{{\mathbb D}}

\newcommand{\F}{{\mathcal F}}

\newcommand{\R}{{\mathbb R}}

\newcommand{\X}{{\mathbf X}}

\newcommand{\bD}{{\overline{\mathbb D}}}

\newcommand{\End}{{\operatorname{End}}}

\renewcommand{\Re}{{\operatorname{Re}\,}}

\newcommand{\inv}{^{-1}}

\renewcommand{\arg}{\,{\operatorname{arg}\,}}

\newcommand{\aand}{{\quad\text{and}\quad}}


\begin{document}
\title[Second Hankel determinant of concave functions]{
On the second Hankel determinant of concave functions
}

\author[R. Ohno]{\noindent Rintaro Ohno}
\email{rohno@ims.is.tohoku.ac.jp}
\address{\newline Graduate School of Information Sciences,
\newline Tohoku University, 
\newline Aoba-ku, Sendai 980-8579, Japan}

\author[T. Sugawa]{\noindent Toshiyuki Sugawa}
\email{sugawa@math.is.tohoku.ac.jp}
\address{\newline Graduate School of Information Sciences,
\newline Tohoku University, 
\newline Aoba-ku, Sendai 980-8579, Japan}

\begin{abstract}
In the present paper, we will discuss the Hankel determinants $H(f)
=a_2a_4-a_3^2$ of order 2 
for normalized concave functions $f(z)=z+a_2z^2+a_3z^3+\dots$
with a pole at $p\in(0,1).$
Here, a meromorphic function is called concave if it maps the unit disk
conformally onto a domain whose complement is convex. 
To this end, we will characterize the coefficient body of order 2
for the class of analytic functions $\varphi(z)$ on $|z|<1$
with $|\varphi|<1$ and $\varphi(p)=p.$
We believe that this is helpful for other extremal problems concerning
$a_2, a_3, a_4$ for normalized concave functions with a pole at $p.$
\end{abstract}
\dedicatory{In dedication to Professor Karl-Joachim Wirths on his 70th birthday.}
\thanks{
The first author was supported by JSPS Grant-in-Aid for JSPS Fellows No. $26\cdot2855$.}
\maketitle

\section{Introduction}
A meromorphic function $f$ on the unit disk
$\D=\{z: |z|<1\}$ of the complex plane $\C$ is called {\it concave}
if $f$ is univalent and if $\C\setminus f(\D)$ is convex.
Such functions are intensively studied by Avkhadiev, Bhowmik, Pommerenke,
Wirths and others in recent years, see \cite{APW06, AW02, AW05, AW07, BPW07}.
For $p\in\D\setminus\{0\},$
we denote by $\Co_p$ the set of concave functions $f$ with a pole
at $p$ normalized by $f(0)=0$ and $f'(0)=1.$
By a suitable rotation, we will assume without loss of generality that
$0<p<1$ in what follows.
Each function $f$ in $\Co_p$ can be expanded in the form
$f(z)=z+a_2z^2+a_3z^3+\dots$ for $|z|<p.$
We sometimes write $a_n=a_n(f)$ to indicate that the coefficients belong to the function $f$.

By $\End(\D)$ we denote the set of analytic endomorphisms (self-maps)
of the unit disk $\D.$
Let $\B_p$ stand for the class of $\varphi\in\End(\D)$ fixing the point $p.$
The first author gave the following characterization of the functions in $\Co_p$  in \cite{O13}.

\begin{Thm}\label{Thm:O}
Let $0<p<1.$
For $f\in\Co_p,$ there exists a $\varphi\in\B_p$ such that
\begin{equation}\label{eq:f'}
f'(z)=\frac{p^2}{(z-p)^2(1-pz)^2}\exp\int_0^z
\frac{-2\varphi(t)}{1-t\varphi(t)}dt,
\quad z\in\D.
\end{equation}
Conversely, for a given $\varphi\in\B_p,$
there exists a function $f\in\Co_p$ satisfying \eqref{eq:f'}.
\end{Thm}

We remark that the condition $\varphi(p)=p$ comes from the demand
that $f'(z)$ should have no residue at $z=p.$

For $a\in\D,$ the M\"obius transformation
$$
T_a(z)=\frac{a-z}{1-\bar az}=-[z,a]
$$
is an analytic involution of $\D$ interchanging $0$ and $a.$
Here $[z,w]=(z-w)/(1-\bar wz)$ denotes the {\it complex pseudo-hyperbolic
distance} introduced by Beardon and Minda \cite{BM04}.
Let $\zeta\in\partial\D.$
Then, the conjugation $\rho_\zeta$ of the rotation $z\mapsto \zeta z$ by $T_p$
is an analytic automorphism of $\D$ contained in $\B_p.$
More explicitly, $\rho_\zeta$ is expressed by
$$
\rho_\zeta(z)=T_p(\zeta T_p(z))
=\frac{(\zeta-p^2)z+(1-\zeta)p}{-(1-\zeta)pz+1-p^2\zeta}
=\alpha_0+\alpha_1z+\alpha_2z^2+\dots,
$$
where
\begin{equation}\label{eq:alpha}
\alpha_0=\frac{(1-\zeta)p}{1-p^2\zeta},
\aand
\alpha_k=\frac{\zeta(1-p^2)^2(1-\zeta)^{k-1}p^{k-1}}{(1-p^2\zeta)^{k+1}},
\quad k=1,2,3,\dots.
\end{equation}
Obviously, $\rho_\zeta$ can be defined for $\zeta\in\bD$
as an analytic endomorphism of $\D.$
Noting the fact that
$$
\frac{-2\rho_\zeta(z)}{1-z\rho_\zeta(z)}
=\frac{(\zeta(z-p)^2-(1-pz)^2)'}{\zeta(z-p)^2-(1-pz)^2},
$$
we see that the function determined by \eqref{eq:f'} with the choice
$\varphi=\rho_\zeta$ is given by
\begin{align}\label{eq:F}
F_\zeta(z)&=\frac{z-T_p(p\zeta)z^2}{(1-z/p)(1-pz)} \\
\notag
&=\frac z{1-p^2\zeta}\left[
\frac 1{1-z/p}-\frac{p^2\zeta}{1-pz}\right] \\
\notag
&=\sum_{n=1}^\infty\frac{1-p^{2n}\zeta}{p^{n-1}(1-p^2\zeta)}z^n
=:\sum_{n=1}^\infty A_n(\zeta)z^n.
\end{align}

Thus we see that the coefficient region $\{a_n(f): f\in\Co_p\}$
contains the set $A_n(\bD)=\{A_n(\zeta): \zeta\in\bD\}.$
We note that $A_n(\bD)$ is the closed disk
$|w-(1-p^{2n+2})/p^{n-1}(1-p^4)|\le (p^2-p^{2n})/p^{n-1}(1-p^4).$
Indeed, Avkhadiev and Wirths \cite{AW07} proved the following.

\begin{Thm}
Let $0<p<1$ and $n\ge2.$
Then
$$
\{a_n(f): f\in\Co_p\}
=A_n(\bD)
=\left\{w: \left|w-\frac{1-p^{2n+2}}{p^{n-1}(1-p^4)}\right|
\le \frac{p^2-p^{2n}}{p^{n-1}(1-p^4)}\right\}.
$$
Moreover, for $f\in\Co_p,$ $a_n(f)\in\partial A_n(\bD)$
if and only if $f=F_\zeta$ for some $\zeta\in\partial\D.$
\end{Thm}

Note that for each $\zeta\in\partial\D,$ 
$T_p(p\zeta)=(1+e^{i\theta})p/(1+p^2)$ for some $\theta\in\R$ and vice versa.
In the present paper, we consider the second Hankel determinant of order $2$
for $f(z)=z+a_2z^2+\dots,$ which is defined by
$$
H(f)=\begin{vmatrix} a_2 & a_3 \\ a_3 & a_4 \end{vmatrix}
=a_2a_4-a_3^2.
$$
Especially, we will take a closer look at the variability region
$H(\Co_p)=\{H(f): f\in\Co_p\}$ for $0<p<1.$
The second Hankel determinant of general order was studied by
Pommerenke \cite{Pom66} and Hayman \cite{Hayman68}
and many others in recent years.
A straightforward computation yields
$$
H(F_\zeta)=A_2(\zeta)A_4(\zeta)-A_3(\zeta)^2
=-\frac{(1-p^2)^2\zeta}{(1-p^2\zeta)^2}
=-\frac{(1-p^2)^2}{p^2}K(p^2\zeta),
$$
where
$$
K(z)=\frac z{(1-z)^2}
$$
is the Koebe function.
Set
\begin{equation}\label{eq:Op}
\Omega_p=\{H(F_\zeta): |\zeta|\le 1\}
=-\frac{(1-p^2)^2}{p^2}K(p^2\bD).
\end{equation}
This set has the following property.

\begin{prop}\label{prop:Omega}
$\Omega_p\subset\Omega_q$ for $0<q<p<1$ and
$$
\bigcup_{0<p<1}\Omega_p=\D\cup\{-1\}
\aand
\bigcap_{0<p<1}\Omega_p=\{-(1+z)^2/4: |z|\le 1\}.
$$
\end{prop}

Note that the set $\{-(1+z)^2/4: |z|\le 1\}$ is a closed Jordan domain, bounded by a cardioid with an inward-pointing cusp at the origin.

By the above observations, we have $\Omega_p\subset H(\Co_p).$
In view of the coefficient regions of $a_n$ for $\Co_p,$
one might suspect that $H(\Co_p)=\Omega_p$ for $0<p<1$
and, in particular, $H(\Co_p)\subset\bD.$
This is, however, not the case.
To state our result, we set
$$
M(p)=\sup\{|H(f)|: f\in\Co_p\}.
$$

\begin{thm}\label{thm:main}
Let $0<p<1.$
Then $M(p)>1.$
Moreover,
$$
\frac1{3p}<M(p)<\frac1{3p}+\frac23.
$$
\end{thm}

In Section 3, we will prove the above proposition and the theorem.
Indeed, we give a description of the variability region of $H(f)$
for $f\in\Co_p$ in Proposition \ref{prop:Phi} below.
As a preliminary, we give an explicit form of the coefficient body
of order 2 for the class $\B_p$ in Section 2.
Our basic idea is to employ an higher-order analogue of Dieusonn\'e's lemma.

\section{Higher-order analogue of Dieudonn\'e's lemma and its application}

We expand a function $\varphi\in\B_p$ in the form
\begin{equation}\label{eq:phi}
\varphi(z)=c_0+c_1z+c_2z^2+\dots,
\quad |z|<1.
\end{equation}
Then the early coefficients of the function $f(z)=z+a_2z^2+a_3z^3+\dots$
determined by \eqref{eq:f'} are given by
\begin{align}
\notag
a_2&=P-c_0, \\
\label{eq:a}
a_3&=P^2+\frac{-c_1+c_0^2-4Pc_0-2}3, \\
\notag
a_4&=P^3+\frac{-c_2+c_0c_1+6c_0-9P-9P^2c_0+3Pc_0^2-3Pc_1}6,
\end{align}
where we put
$$
P=p+\frac1p=\frac{1+p^2}p.
$$
By making use of this type of relations, a coefficient problem
for $\Co_p$ reduces in principle to that of $\B_p.$
Based on this idea, in \cite{OS15a}, the authors solved the extremal problem
on $|a_3-\mu a_2^2|$ for a function $f(z)=z+a_2z^2+a_3z^3+\dots$ in 
$\Co_p$ and a real constant $\mu.$
The key ingredient in \cite{OS15a} was the determination of the coefficient
body $\X_1(\B_p)$ of order $1$ for $\B_p.$
Here, for $n\ge0,$ the coefficient body $\X_n(\F)$ of order $n$ for a
class $\F$ of analytic functions at the origin is defined by
\begin{align*}
\X_n(\F)=&\{(c_0,c_1,\dots,c_n)\in\C^{n+1}: \\
&\quad \quad \varphi(z)=c_0+c_1z+\dots +c_nz^n+O(z^{n+1}) \text{ for some } \varphi\in\F\}.
\end{align*}

The goal of the present section is to show the following description
of the coefficient body $\X_2(\B_p)$ of order 2.

\begin{thm}\label{thm:X2}
Let $0<p<1.$
A triple $(c_0,c_1,c_2)$ of complex numbers is contained in the
coefficient body $\X_2(\B_p)$ if and only if
$$
c_0=P\inv (1-\sigma_0),
\aand
c_1=P^{-2}\big[1+(P^2-2)\sigma_0+\sigma_0^2\big]+P\inv(1-|\sigma_0|^2)\sigma_1
$$
and
\begin{align*}
c_2&=P^{-3}(1-\sigma_0)\big[1+(P^2-2)\sigma_0+\sigma_0^2\big]
-P^{-2}(P^2-2+2\sigma_0)(1-|\sigma_0|^2)\sigma_1 \\
&\qquad
+\varepsilon P\inv(1-|\sigma_0|^2)\overline{\sigma_0}\sigma_1^2
+P\inv(1-|\sigma_0|^2)(1-|\sigma_1|^2)\sigma_2
\end{align*}
for some $\sigma_0,\sigma_1,\sigma_2\in\bD,$
where $P=(1+p^2)/p>2$ and $\varepsilon=|1+p^2\sigma_0|/(1+p^2\sigma_0)
\in\partial\D.$
\end{thm}

Let us now recall Dieudonn\'e's lemma 
(see, for instance, \cite[p.~198]{Duren:univ}).

\begin{lem}[Dieudonn\'e's lemma]
Let $z_0, \tau_0\in\D$ with $|\tau_0|\le|z_0|\ne0.$
Then the variability region of $\tau_1=\psi'(z_0)$ for $\psi\in\End(\D)$
with $\psi(0)=0, \psi(z_0)=\tau_0$ is the closed disk given by
\begin{equation}\label{eq:D}
\left|\tau_1-\frac{\tau_0}{z_0}\right|\le
\frac{|z_0|^2-|\tau_0|^2}{|z_0|(1-|z_0|^2)}.
\end{equation}
\end{lem}

We remark that equality holds in Dieudonn\'e's lemma if and only if
$\psi$ is a finite Blaschke product of degree at most 2 
(cf.~\cite[Theorem 3.6]{CKS12}).
The following result can be regarded as Dieudonn\'e's lemma of the second order
(see \cite[Theorem 3.7]{CKS12}).

\begin{lem}\label{lem:CKS}
Let $z_0,\tau_0\in\D$ with $|\tau_0|<|z_0|\ne0$
and suppose that $\tau_1\in\C$ satisfies \eqref{eq:D}.
Then the variability region of $\tau_2=\psi''(z_0)/2!$ for
$\psi\in\End(\D)$ with $\psi(0)=0, \psi(z_0)=\tau_0$ and $\psi'(z_0)=\tau_1$
is the closed disk described by
$$
\left|
\tau_2-\frac{\tau_1-\tau_0/z_0}{z_0(1-|z_0|^2)}
+\frac{\overline{\tau_0}(\tau_1-\tau_0/z_0)^2}{|z_0|^2-|\tau_0|^2}
\right|
+\frac{|z_0||\tau_1-\tau_0/z_0|^2}{|z_0|^2-|\tau_0|^2}
\le \frac{|z_0|(1-|\tau_0/z_0|^2)}{(1-|z_0|^2)^2}.
$$
\end{lem}

We remark that equality holds precisely when $\psi$ is a finite Blaschke
product of degree at most 3.
In \cite[Theorem 3.7]{CKS12}, the above inequality is stated as a
necessary condition.
For sufficiency, a construction is given in the proof below.

In order to prove Theorem \ref{thm:X2}, we show a preliminary
form of the characterization of $X_2(\B_p).$

\begin{lem}\label{lem:X2}
Let $0<p<1.$
A triple $(c_0,c_1,c_2)$ of complex numbers is contained in the
coefficient body $\X_2(\B_p)$ if and only if
$$
c_0=\frac{p-pw_0}{1-p^2w_0}
\aand
c_1=\frac{(1-p^2)^2w_0+p(1-p^2)(1-|w_0|^2)w_1}{(1-p^2w_0)^2}
$$
and
\begin{align*}
c_2&=\frac{(1-p^2)}{(1-p^2w_0)^3}\left[
p(1-p^2)(1-w_0)w_0
-(1-p^2)(1+p^2w_0)(1-|w_0|^2)w_1\right. \\
&\qquad\left. +p(\overline{w_0}-p^2)(1-|w_0|^2)w_1^2
+p(1-p^2w_0)(1-|w_0|^2)(1-|w_1|^2)w_2
\right]
\end{align*}
for some $w_0,w_1,w_2\in\bD.$
\end{lem}

\begin{proof}
When $\varphi=\rho_\zeta$ for some $\zeta\in\partial\D,$
the coefficients $\alpha_0,\alpha_1,\alpha_2$ are given as
$c_0, c_1, c_2$ with $(w_0,w_1,w_2)=(\zeta,0,0)$ (see \eqref{eq:alpha}).

We next suppose that a function $\varphi(z)=c_0+c_1z+c_2z^2+\dots$ in $\B_p$
is not of the form $\rho_\zeta,~\zeta\in\partial\D.$
Then $\psi=T_p\circ\varphi\circ T_p\in\End(\D)$ satisfies
$\psi(0)=0$ but is not a rotation about $0.$
It is straightforward to check the formulae:
\begin{align*}
\tau_0&:=\psi(p)=T_p(c_0)=\frac{p-c_0}{1-pc_0}, \\
\tau_1&:=\psi'(p)=\frac{c_1}{(1-pc_0)^2}, \\
\tau_2&:=\frac{\psi''(p)}2
=\frac{-(1-pc_0)c_2+pc_1(1-pc_0-c_1)}{(1-p^2)(1-pc_0)^3}.
\end{align*}
Hence,
\begin{align}\notag
c_0&=T_p(\tau_0)=\frac{p-\tau_0}{1-p\tau_0},
\quad (1-pc_0)(1-p\tau_0)=1-p^2, \\
\label{eq:c}
c_1&=(1-pc_0)^2\tau_1=\frac{(1-p^2)^2\tau_1}{(1-p\tau_0)^2}, \\
c_2&=\frac{-(1-p^2)^3(1-p\tau_0)\tau_2
+p(1-p^2)^2\tau_1(1-p\tau_0-\tau_1+p^2\tau_1)}{(1-p\tau_0)^3}.
\notag
\end{align}
By Schwarz's lemma and Dieudonn\'e's lemma (with $z_0=p$), we have
\begin{equation}\label{eq:w01}
w_0:=\frac{\tau_0}p\in\D
\aand
w_1:=\frac{\tau_1-\tau_0/p}{(p^2-|\tau_0|^2)/p(1-p^2)}
=\frac{(1-p^2)(\tau_1-w_0)}{p(1-|w_0|^2)}\in\bD.
\end{equation}
When $|w_1|=1,$ by the remark after Lemma \ref{lem:CKS},
$\psi(z)$ is of the form $z\omega([z,p]),$ where
$\omega(z)=[w_1z,-w_0].$
Then the first three Taylor coefficients of $\psi(z)$ about $z=0$ are given by
$c_0, c_1, c_2$ in \eqref{eq:alpha} with $w_2=0.$

Finally, suppose that $|\omega_1|<1.$
Then, by Lemma \ref{lem:CKS}, we see that
\begin{align}\label{eq:w2}
w_2&:=\left(\tau_2-\frac{\tau_1-\tau_0/p}{p(1-p^2)}
+\frac{\overline{\tau_0}(\tau_1-\tau_0/p)^2}{p^2-|\tau_0|^2}\right)\div
\left(\frac{p(1-|\tau_0/p|^2)}{(1-p^2)^2}
-\frac{p|\tau_1-\tau_0/p|^2}{p^2-|\tau_0|^2}\right) \\
\notag
&~=\frac{(1-p^2)^2\tau_2-(1-|w_0|^2)(1-p\overline{w_0}w_1)w_1}%
{p(1-|w_0|^2)(1-|w_1|^2)}\in\bD.
\end{align}
Here, note that the denominator does not vanish because of $|w_1|<1.$
Conversely, for $w_0, w_1\in\D$ and $w_1\in\bD,$ the function
$\psi(z)=z\omega([z,p])$ fulfills the relations in
\eqref{eq:w01} and \eqref{eq:w2}, where $\omega(z)=[z[w_2z,-w_1],-w_0].$
(Note that this construction shows the sufficiency part of Lemma \ref{lem:CKS}.)
We now obtain
\begin{align*}
\tau_0&=pw_0, \\
\tau_1&=w_0+\frac{p(1-|w_0|^2)w_1}{1-p^2}, \\
\tau_2&=\frac{1-|w_0|^2}{(1-p^2)^2}
\big[(1-p\overline{w_0}w_1)w_1+p(1-|w_1|^2)w_2\big].
\end{align*}
Substitution of these expressions into \eqref{eq:c} proves the lemma.
\end{proof}

\begin{proof}[Proof of Theorem \ref{thm:X2}]
For $w_0, w_1, w_2\in\bD,$ we put
$$
\sigma_0=[w_0,p^2]=\frac{w_0-p^2}{1-p^2w_0},\quad
\sigma_1=\frac{|1-p^2w_0|^2}{(1-p^2w_0)^2}w_1,\quad
\sigma_2=\frac{|1-p^2w_0|^2}{(1-p^2w_0)^2}w_2.
$$
Then $\sigma_j\in\bD$ for $j=0,1,2$ and vice versa.
Noting the elementary relations
$$
w_0=[\sigma_0,-p^2]=\frac{\sigma_0+p^2}{1+p^2\sigma_0},\quad
(1+p^2\sigma_0)(1-p^2w_0)=1-p^4
$$
and
$$
(1-|\sigma_0|^2)|1-p^2w_0|^2=(1-p^4)(1-|w_0|^2),
$$
the formulae of $c_j$ in Lemma \ref{lem:X2}
can be expressed in terms of $\sigma_0,\sigma_1,\sigma_2$
through tedious but straightforward computations.
We finally replace $p+1/p$ by $P$ to prove Theorem \ref{thm:X2}.
\end{proof}

\section{Proof of main results}

By the relations \eqref{eq:a}, we can express $H(f)$ for $f\in\Co_p$ 
in terms of $c_j$'s as follows:
\begin{align}\label{eq:H}
18H(f)&=3(c_0-P)c_2-2c_1^2+(c_0^2-4Pc_0+3P^2-8)c_1 \\
\notag
&\qquad -(c_0^2-Pc_0+1)(2c_0^2-5Pc_0+3P^2+8).
\end{align}
We further substitute the formulae in Theorem \ref{thm:X2} into
\eqref{eq:H} to obtain
\begin{align}\label{eq:Phi}
18 P^3 H(f)
&=
-18P \big[1+(P^2-2)\sigma_0+\sigma_0^2\big] \\
\notag
&\quad +3\big[1-7P^2+2P^4+(3P^2-2)\sigma_0+\sigma_0^2\big]
(1-|\sigma_0|^2)\sigma_1 \\
\notag
&\quad +P \big[2(1-|\sigma_0|^2)
+3\varepsilon\overline{\sigma_0}(P^2-1+\sigma_0)\big](1-|\sigma_0|^2)\sigma_1^2 \\
\notag
&\quad -3P (P^2-1+\sigma_0)(1-|\sigma_0|^2)(1-|\sigma_1|^2)\sigma_2 \\
\notag
&=:\Phi_p(\sigma_0,\sigma_1,\sigma_2),
\end{align}
where $\varepsilon=|1+p^2\sigma_0|/(1+p^2\sigma_0).$
At this stage, we have obtained the following description of the set $H(\Co_p).$

\begin{prop}
\label{prop:Phi}
Let $0<p<1.$
Then the variability region of the second Hankel determinant $H(f)$ of order $2$
for $f\in\Co_p$ is given by
$$
H(\Co_p)=\{\Phi_p(\sigma_0,\sigma_1,\sigma_2)/18P^3: \sigma_0,\sigma_1,\sigma_2\in\bD\}.
$$
\end{prop}

We note that the function $F_\zeta$ given in \eqref{eq:F} corresponds to
the parameters $(\sigma_0,\sigma_1,\sigma_2)=([\zeta,p^2],0,0).$
Since $\Phi_p(\sigma,0,0)=-18P(1+(P^2-2)\sigma+\sigma^2),$
as a by-product, we have the following description of the set $\Omega_p$
defined by \eqref{eq:Op}.

\begin{lem}\label{lem:Omega}
$$
\Omega_p=\{-P^{-2}\big[1+(P^2-2)\sigma+\sigma^2\big]: \sigma\in\bD\},
$$
where $P=(1+p^2)/p,~0<p<1.$
\end{lem}

The description of the Lemma can now be used to show Propositition \ref{prop:Omega}.

\begin{proof}[Proof of Proposition \ref{prop:Omega}]
Put $t=P^2>4$ and write
$$
f_t(z)=-t\inv\big[1+(t-2)z+z^2\big].
$$
Then $\Omega_p=f_t(\bD)$ by Lemma \ref{lem:Omega}.
To show the monotonicity of $\Omega_p,$ it is enough to prove that
$f_t(\D)\subset f_{t'}(\D)$ for $4<t<t'.$
We note that $f_t(z)$ is univalent for each $t>4.$
This is implied by the elementary fact that
$f(z)=z+az^2$ is univalent (indeed, starlike) if and only if $|a|\le 1/2.$
Hence, $\gamma_t(\theta)=f_t(e^{i\theta}),~0\le\theta\le 2\pi,$ gives
a smooth parametrization of the boundary curve of the Jordan domain $\Omega_p.$
We first observe that $f_t(1)=-1$ is stationary with respect to $t.$
In order to show that $f_t(\D)$ is an increasing family of domains,
it is enough to see that the flow $t\mapsto \gamma_t(\theta)$ is outgoing
from $f_{t_0}(\D)$ at the time $t=t_0$ for each $\theta\in(0,2\pi).$
Since an outer normal vector of the boundary curve $\partial\Omega_p$ at 
$\gamma_t(\theta)$ is given by $\gamma_t'(\theta)/i,$ we should show
that $|\arg[i\dot\gamma_t(\theta)/\gamma_t'(\theta)]|<\pi/2,$
where $\dot\gamma_t(\theta)=\partial\gamma_t(\theta)/\partial t.$
By a straightforward computation, we obtain
$$
\Re\frac{\gamma_t'(\theta)/i}{\dot\gamma_t(\theta)}
=\Re\frac{-t\inv(t-2+2e^{i\theta})e^{i\theta}}{t^{-2}(1-e^{i\theta})^2}
=\frac{t(t-2+2\cos\theta)}{4\sin^2(\theta/2)}>0
$$
for $0<\theta<2\pi.$
Thus we have shown the monotonicity of $\Omega_p$ in $p.$
The other assertion easily follows from the facts that
$\lim_{t\to4}f_t(z)=-(1+z)^2/4$ and that
$\lim_{t\to+\infty}f_t(z)=-z.$
\end{proof}

Finally, we prove our main result.

\begin{proof}[Proof of Theorem \ref{thm:main}]
Considering $\sigma_0=t, \sigma_1=-1$ and $\sigma_2 = 0$ 
with $t \in [0,1]$ in   Proposition \ref{prop:Phi} above, we obtain
\begin{align*}
\Phi_p(t,-1,0)=
&
-18P \big[1+(P^2-2)t+t^2\big] \\
&\quad -3\big[1-7P^2+2P^4+(3P^2-2)t+t^2\big]
(1-t^2) \\
&\quad +P\big[2(1-t^2)+3 t (P^2-1+t)\big](1-t^2).
\end{align*}
Setting $h_p(t):=-\Phi_p(t,-1,0)/18P^3$ gives 
\begin{align*}
18P^3h_p(t)=& -(P+3) t^4-(3 P^3+9 P^2-3 P-6) t^3-(6 P^4-21 P^2-17 P) t^2 \\
&\quad+3(7 P^3+3P^2-13 P-2) t +  (6 P^4-21P^2+20 P+3)
\end{align*}
and
\begin{align*}
18P^3h_p'(t)=
& -4(P+3) t^3 -3\left(3 P^3+9 P^2-3 P-6\right) t^2 \\
&\quad +2\left(6 P^4-21 P^2-17 P\right) t
	+3(7P^3+3 P^2-13 P-2),
\end{align*}
which leads to $h_p(1)=1$ and $h_p'(1)=-2(P-2)(P+1)/3 P<0$. 
Thus the function $h_p(t)$ is strictly decreasing at $t=1$ and therefore $h_p(t_0)>1$ for $t_0<1$ sufficiently close to $1.$ Hence, $|H(f)|=h_p(t_0)>1$ for the function $f \in \mathcal{C}o_p$ corresponding to the parameter triple $(t_0,-1,0).$
Thus $M(p)>1$ follows.

To show the inequality $M(p)>1/3p,$ we use the lower estimate
$$
M(p)\ge h_p(\tfrac7{4P})=\frac P3+g(\tfrac1P)
=\frac1{3p}+\frac p3+g(\tfrac1P),
$$
where
$$
g(x)=-\frac{7x}{48}+\frac{143x^2}{72}-\frac{121x^3}{128}
-\frac{427x^4}{1152}+\frac{343x^5}{384}+\frac{5831x^6}{4608}-\frac{2401x^7}{1536}.
$$
We note that $p/3>1/3P.$
It is not difficult to see that $1/3P+g(1/P)=x/3+g(x)>0$ for $x=1/P\in(0,1/2).$
Therefore, $M(p)>1/3p.$

Finally, we show $M(p)<(1+2p)/3p.$
In view of \eqref{eq:Phi}, one can estimate $\Phi_p$ as in
\begin{align*}
|\Phi_p(\sigma_0,\sigma_1,\sigma_2)|&\le
B_0 + B_1x + B_2x^2 + B_3 (1-x^2) \\
&= (B_2-B_3)x^2+B_1x+ B_0+B_3
\end{align*}
with $x=|\sigma_1|,$ where
\begin{align*}
B_0&= 18P \big[1+(P^2-2)y+y^2\big], \\
B_1&= 3\big[1-7P^2+2P^4+(3P^2-2)y+y^2\big](1-y^2), \\
B_2&= P \big[2(1-y^2)+3y(P^2-1+y)\big](1-y^2),  \\
B_3&= 3P (P^2-1+y)(1-y^2)
\end{align*}
with $y=|\sigma_0|.$
Since
$$
\big[2(1-y^2)+3y(P^2-1+y)\big]-3P (P^2-1+y)= (1-y)(1-y-3P^2)\le0
$$
for $P>2,~0\le y\le1,$ we have $B_2-B_3\le 0.$
Thus, we have
\begin{align*}
|\Phi_p(\sigma_0,\sigma_1,\sigma_2)|&\leq B_0 + B_1 + B_3\\
&= 18 P^3 +6P^2 \left( P^2- P-2\right) 2t\\
& \quad -3 P \left(2 P^3+P^2+2 P-4 \right) t^2 +3 \left(3 P^2+P+2\right) t^3 -3 t^4,
\end{align*}
where $t=1-y\in[0,1]$.

Using the inequalities $2t \leq 1+t^2$ and $6t^3-3t^4\le 3t^2$ for $0\le t\le 1,$
we obtain
\begin{align*}
B_0 + B_1 + B_3 
&\leq 6 \left( P^4+2 P^3-2P^2\right)+3 \left(-3 P^3-6 P^2+4 P+1\right) t^2 \\
&\quad +3P \left(3 P+1\right) t^3 
=: G_p(t).
\end{align*}
The function $G_p(t)$ has a maximum at $t=0$ 
(and a minimum at $t=2 (3 P^3+6 P^2-4 P-1)/ (3 P (3 P+1))>1$) 
for all $P>2$. 
Therefore we have
$$
\sup_{\sigma_0,\sigma_1,\sigma_2 \in \overline{\D}}|\Phi_p(\sigma_0,\sigma_1,\sigma_2)|\leq \max_{0\leq t \leq 1} G_p(t) = G_p(0)= 6P^2 \left(P^2+2 P-2\right),
$$
which implies according to Proposition \ref{prop:Phi}
$$
M(p)\leq \frac{P^2+2P-2}{3P}=\frac{1+2p}{3p}-\frac{p(1-p^2)}{1+p^2} < \frac{1}{3p}+\frac{2}{3}.
$$
This completes the proof.
\end{proof}

\def\cprime{$'$} \def\cprime{$'$} \def\cprime{$'$}
\providecommand{\bysame}{\leavevmode\hbox to3em{\hrulefill}\thinspace}
\providecommand{\MR}{\relax\ifhmode\unskip\space\fi MR }
\providecommand{\MRhref}[2]{%
  \href{http://www.ams.org/mathscinet-getitem?mr=#1}{#2}
}
\providecommand{\href}[2]{#2}

\end{document}